\documentclass[runningheads]{llncs}
\usepackage[utf8]{inputenc}
\usepackage[T1]{fontenc}
\usepackage[english]{babel}
\usepackage{amsmath}
\usepackage{amsfonts}
\usepackage{amssymb}
\usepackage{mathtools}
\usepackage{nicefrac}
\usepackage{graphicx}
\usepackage{xcolor}
\usepackage{dsfont}
\usepackage{todonotes}
\definecolor{darkred}{RGB}{150, 0, 0}
\definecolor{darkgreen}{RGB}{0, 150, 0}
\definecolor{darkblue}{RGB}{0, 0, 150}
\usepackage[breaklinks=true,colorlinks,citecolor=darkgreen,linkcolor=darkred,urlcolor=darkblue,bookmarks=false]{hyperref}

\usepackage{cleveref}
\Crefname{theorem}{Thm.}{Thms.}
\Crefname{proposition}{Prop.}{Props.}
\Crefname{lemma}{Lem.}{Lems.}
\Crefname{section}{Sec.}{Secs.}
\Crefname{appendix}{Sec.}{Secs.}
\Crefname{definition}{Def.}{Defs.}

\usepackage{enumerate}
\usepackage{tikz}
\usepackage{booktabs}
\usepackage[inline]{enumitem}
\usepackage{xspace}
\usepackage{subcaption}
\usepackage[binary-units=true]{siunitx}
\usepackage{cite}

\usepackage{listings}
\newcommand{\reals}{\ensuremath{\mathds{R}}}
\newcommand{\Q}{\ensuremath{\mathds{Q}}}
\newcommand{\integers}{\ensuremath{\mathds{Z}}}

\newcommand{\feasregion}{\ensuremath{F}\xspace}
\newcommand{\initcore}{\ensuremath{\mathcal{F}}\xspace}

\newcommand{\define}{\ensuremath{\coloneqq}}

\newcommand{\bnbtree}{\ensuremath{\mathcal T}\xspace}
\newcommand{\bnbnodes}{\ensuremath{V}\xspace}
\newcommand{\bnbedges}{\ensuremath{E}}

\newcommand{\optprob}{\ensuremath{\mathrm{OPT}(f, \feasregion)}\xspace}

\newcommand{\powset}{\ensuremath{\mathcal{P}}}
\newcommand{\core}{\ensuremath{\mathcal{C}}\xspace}
\newcommand{\derived}{\ensuremath{\mathcal{D}}\xspace}
\newcommand{\derive}{\derived}
\newcommand{\assumed}{\ensuremath{\mathcal{A}}\xspace}

\newcommand{\sym}[1]{\ensuremath{\mathcal{S}_{#1}}}
\newcommand{\perm}{\gamma}
\newcommand{\rowperm}{\pi}
\newcommand{\group}{\Gamma}
\newcommand{\inv}[1]{{#1}^{-1}}
\newcommand{\invperm}{\inv{\perm}}
\newcommand{\orb}[2]{\mathrm{orb}_{#1}\left(#2\right)}

\newcommand{\T}{^T}
\newcommand{\sprod}[2]{{#1}\T{#2}}

\newcommand{\ie}{i.e.,\xspace}
\newcommand{\eg}{e.g.,\xspace}
\newcommand{\implication}{\rightsquigarrow}

\newcommand{\children}{\mathop{\delta^{+}}\nolimits}

\newcommand{\dcn}{\mathop{\rm dcn}}

\newcommand{\bradec}{\ensuremath{\mathcal B}}
\newcommand{\bnbdec}{\ensuremath{B}}

\newcommand{\bnbroot}{r}
\newcommand{\emptylist}{(\,)}
\newcommand{\Cconsistent}{\mbox{$\mathcal{C}$-consistent}\xspace}
\newcommand{\solleaf}{\mathop{\rm leaf}}
\newcommand{\lexgeq}{\succeq}
\newcommand{\stricteps}{\epsilon}
\newcommand{\tgeq}[1][\bnbtree]{\succeq_{\stricteps,#1}}

\newcommand{\lexgteps}{\succ_\stricteps}
\newcommand{\tgt}{\succ_{\stricteps,\bnbtree}}

\newcommand{\primbound}{z}
\newcommand{\config}{\ensuremath{\mathfrak{C}}\xspace}
\newcommand{\configdef}{\big(%
  \core,
  \derived,
  g,
  \primbound,
  \bnbtree,
  \stricteps
  \big)}
\newcommand{\complem}[1]{\overline{#1}}

\newcommand{\sgeq}{\trianglerighteq}
\newcommand{\sleq}{\trianglelefteq}

\newcommand{\card}[1]{|#1|}

\DeclareMathOperator{\sign}{sgn}

\newcommand{\down}[1]{\left\lfloor #1 \right\rfloor}

\newcommand{\ourparagraph}[1]{%
  \medskip

  \noindent\emph{#1}
}

\begin{document}

\title{A proof system for certifying symmetry and optimality reasoning
  in integer programming
  \thanks{The work for this article has been conducted within the Research Campus Modal funded by the German Federal Ministry of Education and Research (BMBF grant numbers 05M14ZAM, 05M20ZBM).}}
\titlerunning{A proof system for certifying symmetry and optimality reasoning}

\author{Jasper van Doornmalen\inst{1} \and
  Leon Eifler\inst{2} \and
  Ambros Gleixner\inst{2,3} \and
  Christopher Hojny\inst{1}
}
\authorrunning{J.\ van Doornmalen et al.}

\institute{TU Eindhoven, The Netherlands,
  \email{\{m.j.v.doornmalen,c.hojny\}@tue.nl}
  \and Zuse Institute Berlin, Germany,
  \email{eifler@zib.de}
  \and HTW Berlin, Germany,
  \email{gleixner@htw-berlin.de}
}

\maketitle

\begin{abstract}
  We present a proof system for establishing the correctness of results produced by
  optimization algorithms, with a focus on mixed-integer programming (MIP).
  Our system generalizes the seminal work of Bogaerts, Gocht, McCreesh, and Nordström
  (2022) for binary programs to handle any additional difficulties arising from unbounded
  and continuous variables,
  and covers a broad range of solving techniques, including
  symmetry handling, cutting planes, and presolving reductions.
  Consistency across all decisions that affect the feasible region is achieved by a pair
  of transitive relations on the set of solutions, which relies on the newly introduced
  notion of consistent branching trees.
  Combined with a series of machine-verifiable derivation rules,
  the resulting framework offers practical solutions to enhance the trust in integer
  programming as a methodology for applications where reliability and correctness are key.
\end{abstract}


\section{Introduction}
\label{sec:intro}


In this paper, we are concerned with general optimization problems of the form
$\optprob \define \min\{ f(x) : x \in \feasregion \}$
for a real-valued function~$f\colon \reals^n \to \reals$
and a feasible region~$\feasregion \subseteq \reals^n$ 
defined by a set of constraints.
In particular, we are interested in the solution of mixed-integer programs (MIPs) to
optimality, for which state-of-the-art algorithms are composed of a large array
of different solving techniques~\cite{NemhauserWolsey1988,Achterberg2007,ConfortiCornuejolsZambelli2014,AchterbergBixby2019,HojnyPfetsch2019}. 
Though their correctness is proven \emph{on paper}, with today's tools it seems
prohibitively difficult to verify that they are implemented correctly \emph{in complex
software}.
In fact, there have been recurring reports of incorrect results in different types of solvers~\cite{Brummayer_Lonsing_Biere2010,Cook_et_al2013,Akgun2018MetamorphicTO,GillardSD2019}.

%
Instead of verifying correctness of an implementation, a more viable approach is to supplement the computed result with a \emph{proof of correctness}~\cite{McConnellMehlhornEtAl2011,AlkassarEtAl2011} that 
can be independently and automatically verified. This approach has long been the standard in the SAT community~\cite{HeuleHuntWetzler2014,Cruz-FilipeHeuleEtAl2017,Cruz-FilipeMarquesEtAl2017,Goldberg2003}, but there also exist proof-logging mechanisms for broader classes of problems such as in SMT solving~\cite{de2008proofs,BarbosaEtal2022}, pseudo-boolean optimization~\cite{Gocht_2019_veripb,ElffersGochtMcCreeshNordstrom_2020}, or exact MIP~\cite{VIPR}.

For MIP, no formal proof system has been developed to this date, and the solving techniques that are certifiable in~\cite{VIPR} are severely limited, not covering any methods that remove parts of the feasible region from the search space, \eg due to symmetry arguments. A seminal paper in this direction is~\cite{bogaerts2023certifiedsymbr}, which describes a proof system to handle symmetry- and redundancy-based reasoning for pseudo-boolean optimization.
Our work follows the same paradigm to develop a proof system that addresses the challenges
posed by the presence of unbounded and continuous variables.
As a result, we can certify the correctness of a wide range of static and dynamic, global and local symmetry handling methods, and many other state-of-the-art techniques from the MIP literature.

The overall strategy is to define a \emph{configuration} as a snapshot of the derivations during the solving process.
A trivial starting configuration is iteratively updated by a series of \emph{rules} that are proven to establish valid derivations, eventually certifying optimality or infeasibility of \optprob.
To that end, in \Cref{sec:proofsystem} we introduce the notion of a \emph{consistent branching tree}, which is used to prevent inconsistencies among all derivations that remove parts of the feasible region and to define \emph{validity} of a configuration.
Vaguely speaking, any valid configuration admits the same optimal objective value as \optprob.
In \Cref{sec:rules}, we present a set of validity-preserving transition rules that make it possible to certify optimality or infeasibility \emph{sequentially}.
In \Cref{sec:realization}, we illustrate the framework for many MIP-techniques within LP-based branch-and-bound
and discuss how certificates could be encoded and verified in practice. We conclude with an outlook in \Cref{sec:outlook}.

\section{Trees, Configurations, and Validity}
\label{sec:proofsystem}

The central ingredient of our proof system, as in \cite{bogaerts2023certifiedsymbr}, is a
reflexive and transitive relation on the solution space that defines a consistent
direction for symmetry handling and optimality-based reductions.
While in \cite{bogaerts2023certifiedsymbr} this preorder remains mostly abstract, we
provide a concrete construction designed to resolve the additional difficulties
introduced by the presence of continuous and unbounded variables, and to capture the
degrees of freedom due to locality in a branch-and-bound process.
We require the following notational conventions.

In the abstract setting of \Cref{sec:proofsystem,sec:rules}, we consider constraints to be
arbitrary subsets $C\in\powset(\reals^n)$, the power set of $\reals^n$;
$\complem{C} = \reals^n \setminus C$ denotes its complement.
For a set of constraints $\core\subseteq\powset(\reals^n)$, we write $\bigcap\core$ as a
shorthand for $\bigcap_{C\in\core}C$.
For $\tau = (i_1, \dots, i_{k})\in\{\pm
1,\ldots,\pm n\}^{k}$, we define by slight abuse of notation 
\[ \tau\colon\reals^n\to\reals^k, \tau(x) \define \big( \sign(i_1)\,x_{|i_1|}, \dots, \sign(i_{k})\,x_{|i_{k}|} \big), \]
where $\sign(\cdot)$ denotes the sign function,
and $\tau(C) \define \{ \tau(x) : x \in C \}$ for $C \subseteq \reals^n$.
We write $\tau \sleq \tau' = (i'_1, \dots, i'_{k'})$
if $i_j=i'_j$ for all $1\leq j \leq k \leq k'$.
%
We write~$\children(v)$ for the set of children of a node $v$ in a directed graph, and $[a] \define \{j \in \integers_+ : j \leq a \}$.

\begin{definition}
  Let $\core\subseteq\powset(\reals^n)$,
  then we call $\bnbtree = (\bnbnodes, \bnbedges, \bnbdec, \sigma)$
  a \emph{\Cconsistent branching tree}
  if the following conditions hold:
\begin{enumerate}[label={(T\arabic*)}, ref={T\arabic*},left=0pt]
\item
\label{consistent:arborescence}
$(\bnbnodes, \bnbedges)$ is an arborescence with $1 \leq \card{\bnbnodes} < \infty$
and root $\bnbroot\in\bnbnodes$.
\item
\label{consistent:branching}
Each~$v \in \bnbnodes$ has attached a branching constraint $\bnbdec_v \subseteq \reals^n$, and $\bnbdec_\bnbroot = \reals^n$.
\item 
\label{consistent:covering}
For all $u \in \bnbnodes$ with $\children(u) \neq \emptyset$ we have
$\bigcap\core \subseteq \bigcup_{v \in \children(u)} \bnbdec_v$.
\item 
\label{consistent:dupl}
Each~$v \in \bnbnodes$ has attached a 
list $\sigma_v\in\{\pm 1,\ldots,\pm n\}^{\ell_v}$, $\ell_v \in \{0,1,\ldots,n\}$,
of signed variable indices
without duplicates in $|(\sigma_v)_1|,\ldots,|(\sigma_v)_{\ell_v}|$.
\item
\label{consistent:growing}
For all $u \in \bnbnodes$, $v \in \children(u)$,
we have $\sigma_u \sleq \sigma_v$.
\item
\label{consistent:bounded}
For all $v \in \bnbnodes$ and $i\in[\ell_v]$,
we have
$
\max\{ \sigma_v(x)_i : x \in \bigcap \core \} < \infty.
$
\item
\label{consistent:packing}
For all $u\in\bnbnodes$ and $v,w \in \children(u)$, $v\not=w$, we have
$
\sigma_u(\bnbdec_v) \cap \sigma_u(\bnbdec_w) = \emptyset
$.
\end{enumerate}
\end{definition}
For a node~$u \in \bnbnodes$ with~$\children(u) \neq \emptyset$,
\eqref{consistent:covering} and~\eqref{consistent:packing} ensure that
the children of $u$ partition the feasible region of $u$;
in particular, for every~$x \in \bigcap\core$, there is a unique leaf
of~$(\bnbnodes,\bnbedges)$ whose feasible region contains~$x$, denoted
by $\solleaf(x)$, see \Cref{lemma:uniquechild,lemma:uniqueleaf} in
\Cref{sec:auxresults}.
%
Hence, for~$x,y\in\bigcap\core$, there exists a \emph{deepest
common node}~$w\eqqcolon\dcn(x,y)$ with $x$ and $y$ both contained in
the feasible region of $w$.
%
The partitioning condition~\eqref{consistent:packing} is stronger and
further implies that all $u\in V$ with $|\children(u)|\geq 2$ have
nontrivial $\sigma_u$.
By~\eqref{consistent:growing}, also all successors $v$ have
$\ell_v\geq \ell_u\geq 1$.
The boundedness condition~\eqref{consistent:bounded} holds if all
variables contained in at least one $\sigma_v$ have finite upper
(lower) bound if contained with positive (negative) index.
%


\Cref{lem:preorder} in \Cref{sec:auxresults} shows that the relations~$\tgeq$ and $\tgt$
defined next constitute a preorder resp. a strict order on the set of solutions
in~$\bigcap\core$:
%
\begin{definition}
\label{def:constree}
Let $\bnbtree$ be a \Cconsistent branching tree, $\stricteps > 0$, and~$x, y \in \reals^n$.
Let~$\lexgeq$ and $\lexgteps$ denote the standard resp.~a strict lexicographic order given by
\vspace{-1.5ex}
\begin{align*}
  x \lexgeq y :\Leftrightarrow\;
  &x=y \lor x_i > y_i \text{ for the smallest index $i$ with $x_i\not=y_i$},\\
  x \lexgteps y :\Leftrightarrow\;
  &x\not=y \land x_i \geq y_i + \stricteps \text{ for the smallest index $i$ with $x_i\not=y_i$}.
  \shortintertext{Using $v=\dcn(x,y)$, we then define the relations $\tgeq$ and $\tgt$ over $\bigcap\core$ by}
  x \tgeq y :\Leftrightarrow\;
     &\sigma_{v}(x) = \sigma_{v}(y) \lor \sigma_{v}(x) \lexgteps \sigma_{v}(y),\\
  x \tgt y :\Leftrightarrow\;
     &\sigma_{v}(x) \lexgteps \sigma_{v}(y).
\end{align*}
\end{definition}
With this, we arrive at the following central notion of a \emph{configuration}:
\begin{definition}
\label{def:validconf}
Let $\core,\derived\subseteq\powset(\reals^n)$, $g\colon\reals^n \to \reals, z \in \reals \cup \{\infty\}, \epsilon > 0$,
and let $\bnbtree$\! be a branching tree.
A configuration $\configdef$ is called \emph{$(\feasregion, f)$-valid} if:
\begin{enumerate}[label={(V\arabic*)}, ref={V\arabic*},left=0pt]
\item 
\label{cond:consistent}
$\bnbtree$ is \Cconsistent and $\stricteps > 0$.
\item 
\label{cond:obj}
If $\primbound < \infty$, there exists an $x \in \feasregion$ with $f(x) \leq \primbound$.
\item 
\label{cond:feasC}
For all $\hat\primbound < \primbound$,
there exists an $x \in \feasregion$
with~$f(x) \leq \hat\primbound$
if and only if
there exists an~$x \in \bigcap\core$
with~$g(x) \leq \hat\primbound$.
\item
\label{cond:derive}
For every $x \in \bigcap\core$ with $g(x) < \primbound$, there exists a $y \in \bigcap (\core \cup \derived)$ with $y \tgeq x$ and $g(y) \leq g(x)$.
\end{enumerate}
\end{definition}
As in \cite{bogaerts2023certifiedsymbr}, we call $\core$ the set of \emph{core} and
$\derived$ the set of \emph{derived constraints}.
\Cref{thm:goalconfig} states that the goal---to produce a valid configuration containing the
contradiction $\emptyset\in\derived$---yields a certificate of optimality (if
$z < \infty$) or infeasibility (if $z=\infty$).
While it seems out of reach to produce such a configuration and check its validity in one
step, \Cref{thm:initvalid} provides a trivially valid starting configuration~$\config^0$.
A proof in our system then becomes a \emph{sequence of valid configurations} starting
at~$\config^0$ and transitioning to richer configurations by iteratively applying a set of
\emph{machine-verifiable rules} that are \emph{proven to preserve validity} in~\Cref{sec:rules}.

\begin{theorem}
  \label{thm:goalconfig}
  Let $\configdef$ be an $(\feasregion, f)$-valid configuration with~$\emptyset \in
  \derived$.
  Then there exists no solution $x \in \feasregion$ with $f(x) < \primbound$.
\end{theorem}

\begin{proof}
  Suppose there is an $(\feasregion, f)$-valid configuration
  $\configdef$ with $\emptyset \in \derived$, and $x \in \feasregion$ with
  $f(x) = \hat\primbound < \primbound$.
  By~\eqref{cond:feasC}, there is~$\hat x \in \bigcap\core$ with $g(\hat x) \leq
  \hat\primbound < \primbound$, and by~\eqref{cond:derive}, there exists a $y \in
  \bigcap(\core \cup \derived)$.
  This contradicts~$\bigcap(\core \cup \derived) \subseteq \bigcap\derived = \emptyset$.
  \qed
\end{proof}

\begin{theorem}
  \label{thm:initvalid}
  Let $\initcore\subseteq\powset(\reals^n)$, $F = \bigcap\initcore$, and
  \mbox{$f\colon\reals^n\to\reals$}.
  Then the \emph{initial configuration} $\config^0\define\big(\initcore, \emptyset, f, \infty, (\{\bnbroot\},
  \emptyset, \{\bnbroot\mapsto\reals^n\}, \{\bnbroot\mapsto\emptylist\}),1\big)$ is
  $(\feasregion, f)$-valid.
\end{theorem}

\begin{proof}
  \eqref{cond:obj} holds due to $\primbound = \infty$ and
  \eqref{cond:feasC} because~$\bigcap\core = \bigcap\initcore =
  \feasregion$ and $g=f$.
  Since $\derived=\emptyset$, we can always choose $y=x$ to satisfy \eqref{cond:derive}.
  Finally, the initial tree~$(\{\bnbroot\}, \emptyset, \{\bnbroot\mapsto\reals^n\},
  \{\bnbroot\mapsto\emptylist\})$ consists only of a root node~$\bnbroot$ with
  $\bnbdec_r=\reals^n$ and empty $\sigma_r = \emptylist$.
  This tree trivially conforms to \Cref{def:constree}, and $\epsilon=1>0$.
  \qed
\end{proof}

\section{Validity-Preserving Transition Rules}
\label{sec:rules}

The first two rules feature a special type of constraint called an
\emph{implication}.
%
Let $\assumed\subseteq\powset(\reals^n)$ be a collection of \emph{assumptions} and $C\subseteq
\reals^n$, then we define $[\assumed \implication C] \define
\complem{\bigcap \assumed} \cup C$, \ie $x \in [\assumed \implication C]$ if and only if $(\, x
\in \bigcap \assumed \Rightarrow x \in C\, )$. \linebreak In particular, any constraint $C$ can be written as
the implication $[\reals^n \implication C]$.


\ourparagraph{Implicational derivation rule.}
This first rule allows to derive constraints that preserve
all feasible, improving solutions in a subset of $\reals^n$, i.e., under a (possibly
empty) set of assumptions $\assumed$:
\begin{theorem}
  \label{thm:implic}
  Let $\configdef$ be an $(\feasregion, f)$-valid configuration.
  Let $C\subseteq\reals^n$ and $\assumed\subseteq\powset(\reals^n)$.
  Then $\big( \core, \derived \cup \{[\assumed \implication C]\}, g, \primbound, \bnbtree,
  \stricteps \big)$ is also $(\feasregion, f)$-valid, if
  \begin{align}
    \label{cond:implic}
    C \supseteq \bigcap (\core \cup \derived \cup \assumed)
    \cap
    \{ x \in \reals^n : g(x) < \primbound \}.
  \end{align}
\end{theorem}

\ourparagraph{Resolution rule.}
A particular form of the implicational derivation rule allows to
derive new constraints by a disjunctive argument.

\begin{corollary}
  \label{cor:resolution}
Let $\configdef$ be an $(\feasregion, f)$-valid configuration.
Consider sets $A_1$,$A_2$,$C_1$,$C_2\subseteq\reals^n$,
and let $\assumed_1,\assumed_2$
be sets of constraints with
\begin{equation}
  \label{cond:resolution}
  \left[\left(\assumed_1 \cup \{A_1\}\right) \implication C_1\right],
  \left[\left(\assumed_2 \cup \{A_2\}\right) \implication C_2\right]
  \in\core\cup\derived
  \,\land\,
  A_1 \cup A_2 \supseteq \bigcap(\core\cup\derived).
\end{equation}
Then $\big( \core, \derived \cup \{ \left[(\assumed_1 \cup \assumed_2) \implication (C_1 \cup C_2)\right] \}, g, \primbound, \bnbtree,
  \stricteps \big)$ is also $(\feasregion, f)$-valid.
\end{corollary}
A typical example for resolution is the case where
$\assumed_1=\assumed_2$ contains the branching decisions up to a parent node
in a search tree with exactly two children, and $C_1=C_2$ is a constraint
derived previously for both children.
%

\ourparagraph{Objective bound update rule.}
When an improving solution is known, this rule allows to update the value of $z$ in a
configuration:

\begin{theorem}
  \label{thm:objbound}
  Let $\configdef$ be an $(\feasregion, f)$-valid configuration and
  $x\in\bigcap\core$ with $g(x) = \primbound' < \primbound$.
  Then $\big( \core, \derived, g, \primbound', \bnbtree, \stricteps \big)$ is also
  $(\feasregion, f)$-valid.
\end{theorem}

\ourparagraph{Objective function update rule.}
If the constraints imply equations, it may be desirable to substitute variables and update
the objective function accordingly:
%

\begin{theorem}
  \label{thm:objfunc}
  Let $\configdef$ be an $(\feasregion, f)$-valid configuration and
  $g'\colon \reals^n\to\reals^n$.
  If
  $
  %
  \bigcap \core \subseteq \{ x \in \reals^n : g'(x) = g(x) \},
  $
  then $\big( \core, \derived, g', \primbound, \bnbtree, \stricteps \big)$ is also
  $(\feasregion, f)$-valid.
\end{theorem}

\ourparagraph{Redundance-based strengthening rule.}
The implicational derivation rule allows the removal of suboptimal solutions only once
$\primbound < \infty$ has been established by the objective bound update rule.
The following rule allows the removal of feasible solutions
independently of the objective bound if a so-called \emph{witness}
$\omega$ is known:

\begin{theorem}
  \label{thm:redundance}
  Let $\config=\configdef$ be an $(\feasregion, f)$-valid configuration and let $C \subseteq
  \reals^n$.
  If there exists an $\omega\colon\reals^n\to\reals^n$ such that
  \begin{equation}
    \label{eq:redruleimp}
    \begin{aligned}
      x \in \bigcap (\core \cup \derived)
      \land
      x\not\in C 
      \Rightarrow
      \;\omega(x) \in \bigcap (\core \cup \derived \cup \{C\})
      \;\land&\;\\
      \omega(x) \tgeq x
      \;\land&\;
      g(\omega(x)) \leq g(x)
    \end{aligned}
  \end{equation}
  holds for all $x\in\reals^n$,
  then $\big( \core, \derived \cup \{C\}, g, \primbound, \bnbtree,
  \stricteps \big)$ is also $(\feasregion, f)$-valid.
\end{theorem}

\begin{proof}
  \eqref{cond:consistent},
  \eqref{cond:obj}, and~\eqref{cond:feasC} are invariant.
  To show that also~\eqref{cond:derive} continues to hold,
  let~$x \in \bigcap\core$ with~$g(x) < z$.
  We need to prove that there is a~$y \in \bigcap(\core\cup\derive\cup\{C\})$ with~$y
  \tgeq x$ and~$g(y) \leq g(x)$.
  By \eqref{cond:derive} for \config, there is~$x' \in \bigcap(\core\cup\derive)$
  with~$x' \tgeq x$ and~$g(x') \leq g(x)$.
  If~${x'\in C}$, we can choose $y=x'$.
  Otherwise, choose $y=\omega(x')$.
  Then \eqref{eq:redruleimp} guarantees that
  $
  y \in \bigcap (\core \cup \derived \cup \{C\})
  $,
  $y \tgeq x' \tgeq x$,
  and $g(y) \leq g(x') \leq g(x)$.
  By transitivity of $\tgeq$, this concludes the proof.
  \qed
\end{proof}

\ourparagraph{Dominance-based strengthening rule.}
\Cref{thm:redundance} 
is limited to cases where a single witness can repair
removed solutions in one step such that they satisfy all core and
derived constraints at once. The next rule allows more complex reasoning:

\begin{theorem}
  \label{thm:dominance}
  Let $\config=\configdef$ be an $(\feasregion, f)$-valid configuration and let $C \subseteq \reals^n$.
  If there exists an $\omega\colon\reals^n\to\reals^n$ such that
  \begin{equation}
    \label{eq:symruleimp}
    \begin{aligned}
      x \in \bigcap (\core \cup \derived)
      \land
      x\not\in C 
      \Rightarrow
      \;\omega(x) \in \bigcap\core
      \;\land\;
      \;\omega(x) \tgt x
      \;\land\;
      g(\omega(x)) \leq g(x)
    \end{aligned}
  \end{equation}
  holds for all $x\in\reals^n$,
  then $\big( \core, \derived \cup \{C\}, g, \primbound, \bnbtree,
  \stricteps \big)$ is also $(\feasregion, f)$-valid.
\end{theorem}
\begin{proof}
  Note that adding~$C$ to~$\derive$ keeps \eqref{cond:consistent}, \eqref{cond:obj},
  and~\eqref{cond:feasC} invariant.
  To prove that~\eqref{cond:derive} holds if one
  replaces~$\derive$ by~$\derive' = \derive \cup \{C\}$,
  let~$x^0 \in \bigcap\core$ with~$g(x^0) < z$.
  We need to show that there is~$y \in \bigcap(\core\cup\derive')$
  with~$y \tgeq x^0$ and~$g(y) \leq g(x^0)$.
  We establish this by creating a (potentially infinite) sequence~$(x^i)_{i = 1}^k$ with
  \[
    x^i \define
    \begin{cases}
      y^i, & \text{if } i \text{ is odd for some~$y^i$
           from~\eqref{cond:derive} for~$x=x^{i-1}$ in~$\config$},\\
      \omega(x^{i-1}), & \text{if~$i$ is even},
    \end{cases}
  \]
  where~$k = \inf\{i \in \integers_+ : x^i \in
  \bigcap(\core\cup\derive\cup\{C\})\}$.
  This sequence is well-defined, which can be shown via induction by alternatingly applying~\eqref{cond:derive} and~\eqref{eq:symruleimp}.

  It suffices to prove~$k < \infty$ as the alternating application
  of~\eqref{cond:derive} and~\eqref{eq:symruleimp}
  guarantees~$x^k \tgeq x^{k-1} \tgeq \dots \tgeq x^0$
  and~$g(x^k) \leq g(x^{k-1}) \leq \dots \leq g(x^0)$, so we can choose $y=x^k$.
  To show finiteness, first note that~\eqref{cond:derive} and~\eqref{eq:symruleimp} imply
  \[
    x^{2i}
    \tgt
    x^{2i-1}
    \tgeq
    x^{2i-2}
    \tgt
    \dots
    \tgeq
    x^2
    \tgt
    x^1
    \tgeq
    x^0
  \]
  for all~$i \in [\nicefrac{k}{2}] = \{j\in\integers_+:j \leq \nicefrac{k}{2}\}$.
  Thus, \Cref{lem:doubleTrans} in the appendix shows $x^{2i} \tgt
  x^{2i-2}$ for all~$i \in [\nicefrac{k}{2}]$,
  i.e., for each~$i \in [\nicefrac{k}{2}]$, there is a~$u_i \in \bnbnodes$
  with~$\sigma_{u_i}(x^{2i}) \lexgteps \sigma_{u_i}(x^{2i-2})$.

  For the sake of contradiction, suppose~$k = \infty$.
  Since~$V$ in $\bnbtree$ is finite, also $\{\sigma_v : v \in
  \bnbnodes\}$ is finite.
  Transitivity of~$\tgt$ and~$k = \infty$ then implies that there
  exists~$u \in \bnbnodes$ and a subsequence~$(x^{i_j})_{j = 1}^\infty$
  with~$\sigma_u(x^{i_{j+1}}) \tgt \sigma_u(x^{i_j})$
  and\linebreak $i_{j+1} > i_j\in 2\integers$ for all $j\in\integers_+$.
  %
  %
  Hence, for all~$j \in \integers_+$, there is~$t_j \in [\ell_u]$ such
  that
  \[
  \sigma_u(x^{i_{j+1}})_{t_j} \geq \sigma_u(x^{i_j})_{t_j} + \stricteps
  \,\land\,
  \sigma_u(x^{i_{j+1}})_s = \sigma_u(x^{i_j})_s
  \text{ for all } s \in [t_j-1].
  \]
  Since~$\ell_u$ is finite, there exists an index~$t$
  that infinitely often determines the lexicographic difference.
  Among all such indices~$t$, let~$t'$ be the smallest one.
  
  By \eqref{consistent:bounded}, all entries in~$\sigma_u(x)$, $x\in\bigcap\core$, are
  bounded from above.
  Hence, there must exist infinitely many~$j$
  with~$\sigma_u(x^{i_{j+1}})_{t'} < \sigma_u(x^{i_j})_{t'}$;
  otherwise, we would have $\lim_{j\to\infty} \sigma_u(x^{i_j})_{t'} = \infty$.
  For all such $j$, due to~$\sigma_u(x^{i_{j+1}}) \tgt
  \sigma_u(x^{i_j})$, there must exist $s \in [t'-1]$ with
  $\sigma_u(x^{i_{j+1}})_s \geq \sigma_u(x^{i_j})_s + \stricteps$.
  This is a contradiction to~$t'$ being minimal.
  Consequently, $k < \infty$.
  This concludes the proof.
  \qed
\end{proof}
In comparison to redundance-based strengthening, dominance-based
strengthening is both weaker in the sense that $\omega(x)$ does not
need to satisfy $\derived$ and $C$ directly, and stricter in the sense that
$\omega$ must result in an increase w.r.t. the strict order~$\tgt$.
Note that in both rules, $C$ may be an implication $[\assumed \implication
  C]$, e.g., a locally valid symmetry-breaking constraint for some node $v\in\bnbnodes$.

\ourparagraph{Epsilon shrinkage rule.}
The presence of $\stricteps$ in \eqref{cond:derive} ensures consistent progress in
the proof of \Cref{thm:dominance}.
It needs to be fixed for each dominance-based strengthening, but intermediately, $\stricteps$ can be decreased to an
arbitrarily small, positive value:

\begin{theorem}
  \label{thm:epsupdate}
  Let $\configdef$ be an $(\feasregion, f)$-valid configuration.
  If $0 < \stricteps' < \stricteps$, then $\big( \core, \derived, g,
  \primbound, \bnbtree, \stricteps' \big)$ is also $(\feasregion,
  f)$-valid.
\end{theorem}

\ourparagraph{Transfer rule.}
Derived constraints can be upgraded to core constraints. 
This may restrict future applications of dominance-based strengthening, but can be useful to
facilitate an objective function update or the tree exchange rule below:

\begin{theorem}
  \label{thm:transfer}
  Let $\configdef$ be an $(\feasregion, f)$-valid configuration and let
  $C\in\derived$.
  Then $\big( \core \cup \{ C\}, \derived \setminus \{ C\}, g, \primbound, \bnbtree,
  \stricteps \big)$ is also $(\feasregion, f)$-valid.
\end{theorem}

\ourparagraph{Deletion rule.}
%
This rule allows to remove derived and redundant core constraints:

\begin{theorem}
  \label{thm:delete}
  Let $\configdef$ be an $(\feasregion, f)$-valid configuration, let
  $\core'\subseteq\core$ and $\derived'\subseteq\derived$.
  Then $\big( \core', \derived', g, \primbound, \bnbtree, \stricteps \big)$ is also
  $(\feasregion, f)$-valid, if 
  \begin{enumerate*}[label=(\alph*)]
  \item $\core'=\core$, or
  \item \mbox{$\core' = \core \setminus \{ C \}$} for 
    $C\supseteq\bigcap\core'$,
    or
  \item \mbox{$\core' = \core \setminus \{ C \}$},\! $\sigma_v=\emptylist$ for all $v\in\bnbnodes$,\! 
    and $C\subseteq\reals^n$ is derivable from 
    $\big(\core', \emptyset, g, z, \bnbtree, \stricteps \big)$ by redundance-based
    strengthening.
  \end{enumerate*}
\end{theorem}

\ourparagraph{Tree exchange rule.}
The preorder $\tgeq$ based on the tree $\bnbtree$ is essential in
guaranteeing that a valid configuration cannot contain contradictory
derivations across different applications of redundance- or
dominance-based strengthening.
The following rule allows to install a new tree before any constraints
have been derived, or when all derived constraints have been deleted
or transferred to \core:

\begin{theorem}
  \label{thm:treeexchange}
  Let $\big( \core, \emptyset, g, \primbound, \bnbtree, \stricteps
  \big)$ be an $(\feasregion, f)$-valid configuration and let
  $\bnbtree'$ be a \Cconsistent branching tree.
  Then $\big( \core, \emptyset, g, \primbound, \bnbtree', \stricteps
  \big)$ is also $(\feasregion, f)$-valid.
\end{theorem}

\ourparagraph{Dimension extension rule.}
Last, but not least, the dimension of the problem can be increased by introducing new
variables that (initially) do not feature in any of the constraints,
the objective function, or the preorder-defining tree:

\begin{theorem}
  \label{thm:dimext}
  Let $\configdef$ be an $(\feasregion, f)$-valid configuration and let
  $\core' = \{ C\times\reals : C\in\core \}$,
  $\derived' = \{ C\times\reals : C\in\derived \}$,
  $g'\colon\reals^{n+1}\to\reals, g'(x_1,\ldots,x_{n+1})=g(x_1,\ldots,x_n)$,
  and $\bnbtree' = (V,E,B',\sigma)$ with $B'_v = \bnbdec_v\times\reals$ for all $v\in V$.
  Then $\big( \core', \derived', g', \primbound, \bnbtree',
  \stricteps \big)$ is also $(\feasregion, f)$-valid.
\end{theorem}
This allows to create extended formulations that help to derive new
constraints, typically by redundance-based strengthening, as discussed
in the next section.

\section{Certificates for mixed-integer programs}
\label{sec:realization}


Let (P) denote a mixed-integer program
$
\min \{ \sprod{c}{x} : Ax \leq b,\; x \in \integers^p \times \reals^{n-p}\},
$
where~$m,p,n \geq 0$ are integers,~$p \leq n$, $A \in
\Q^{m\times n}$, $b \in \Q^m$, and~$c \in \Q^n$.
In the following section,
we show how many popular MIP techniques can be certified within the framework of \Cref{sec:proofsystem,sec:rules}
and automatically verified.

Constraints in the abstract proof system were left to
be arbitrary subsets of~$\reals^n$.
In the MIP setting, the sets of core constraints~$\core$ and
derived constraints~$\derived$ consist of integrality restrictions on variables and linear inequalities or implications $[\mathcal A \implication C]$, where~$C$ and each $A \in \mathcal A$ is a linear inequality.
%
In the initial configuration of \Cref{thm:initvalid}, $\core=\initcore$ contains the rows of $Ax \le b$, as well as $x_j \in \integers$ for all $j \in [p]$.
The objective function is a linear function $f(x) = \sprod{c}{x}$.
In $\bnbtree$, the branching decisions $\bnbdec_v$, $v\in\bnbnodes$,
are also linear inequalities.

While a configuration does not explicitly contain the value of a dual bound, it is
natural that an LP-based branch-and-bound solver would maintain in $\derived$ a set of
implications $[\assumed \implication g(x)\geq d_\assumed]$ per subproblem $\assumed$,
derived by the implicational derivation rule using aggregation of constraints with dual multipliers, and apply the resolution rule
in order to derive globally valid dual bounds from these, much like in~\cite{VIPR}.

\subsection{Presolving, Propagation, and Branch-and-Cut}

A \emph{cutting plane} for (P) is an inequality $\sprod{\alpha}{x} \le \beta$ that is valid for the feasible region of (P), while separating some infeasible point $x^*$. Cuts that are not valid globally can be represented as an implication $[\mathcal A \implication C]$, where $\mathcal{A}$ contains the local constraints for which the cut is valid. Hence, it is clear that any cutting plane can---in the abstract proof system---be certified using the implication rule presented in \Cref{thm:implic}. However, the complexity of this in general as hard as solving (P)~\cite{GroetschelLovaszShrijver1988}. Still, for a wide variety of cutting planes, efficient verification is possible if some additional information is provided.

\ourparagraph{Aggregation.} If there exist multipliers $\lambda_i \ge 0$ such that $C \equiv \sum_{i}\lambda_i C_i$ for some linear constraints $C_i$, then $C$ is redundant, and therefore the implication is trivially satisfied. It is clear that if the indices and multipliers are specified, a verifier can efficiently check that $C\equiv \sum_{i} \lambda_i C_i$ holds. If the multipliers are not specified, a verifier can solve an LP to find the correct multipliers and indices~\cite{EiflerGleixner2023}. Therefore, we can efficiently verify the validity of cuts that are derived by aggregation.

\ourparagraph{Disjunctive cuts.} A disjunctive cut is a cut $\sprod{\alpha}{x} \le \beta$ for which there exists a disjunction $(\sprod{\pi}{x} \le \pi_0) \lor (\sprod{\pi}{x} \ge \pi_0 + 1)$ such that the inequality is valid for both $\{x : Ax \le b, \sprod{\pi}{x} \le \pi_0\}$ and $\{x : Ax \le b, \sprod{\pi}{x} \ge \pi_0 +1\}$~\cite{Jereslow1977, Balas1979, CookKananShrijver1990}. Moreover, for all variables $x_i$ that are not integer $\pi_i$ is required to be $0$. The verification for this type of cut is possible if the correct disjunction is provided in the certificate, which allows for a proof by aggregation for both sides of the disjunction. Afterwards, the validity of the complete cut can be verified by applying the resolution rule, see~\cite{EiflerGleixner2023} for details.
As shown in~\cite{CournejolsLi2001}, an efficient procedure for verifying disjunctive cuts allows to verify many families of cuts including Chv\'atal-Gomory cuts~\cite{Chvatal1973}, Gomory mixed-integer cuts~\cite{Gomory1960}, mixed-integer rounding cuts\cite{NemhauserWolsey1990}, or lift-and-project cuts~\cite{BalasCeriaCournejols1993}.
In the appendix, we also provide examples for verification of simple knapsack~\cite{Balas1975} and flowcover cuts~\cite{Padberg1985}.

\ourparagraph{Extended Formulations.} There are other types of cuts that cannot be characterized as a split cut, \eg reformulation-linearization cuts~\cite{SheraliAdams1990}. However, using the dimension extension rule in combination with conic combinations of constraints, the needed extended formulation can be produced in the certificate.

\ourparagraph{Presolving and Propagation.}
Presolving is a vital part of solving MIPs~\cite{AchterbergBixby2019,Savelsbergh1994,Fuegenschuh2005}, and often certifying a presolving step is straightforward. For example, if a variable bound is tightened by \emph{constraint propagation}~\cite{Savelsbergh1994}, the new variable bound can be verified by aggregation of existing constraints.
Concretely, given some constraint $\sprod{\alpha}{x} \le \beta$, and assume finite bounds on all variables $\ell_i \le x_i \le u_i,  i \in [n]$, an activity-based upper bound for $x_i$ (assuming $\alpha_i > 0$) can be computed as
$
x_i \le \nicefrac{(\beta - \sum_{\alpha_j \ge 0, j \neq i}\alpha_j \ell_j - \sum_{\alpha_j < 0, j \neq i}\alpha_ju_j)}{\alpha_i}
$.
An automatically verifiable proof by implicational derivation is to aggregate the original inequality $\sprod{\alpha}{x} \le \beta$ with $-\alpha_j(x_j \le u_j)$ for $\alpha_j < 0$ and with $\alpha_j(x_j \ge \ell_j)$ for $\alpha_j \ge 0$. Scaling the resulting inequality by~$1/\alpha_i$ yields a proof of the new bound. 
Similar arguments can be made for many other presolving reductions, e.g., probing~\cite{AchterbergBixby2019}. 
In the following, we mention several key presolving techniques that actively exclude feasible solutions of the original problem \optprob.

\ourparagraph{Reduced cost fixing.}
This core technique dates back to the seminal paper of Dantzig, Fulkerson, and Johnson on the traveling salesman problem~\cite{dantzig1954tsp}, and allows to tighten the bounds of a variable using information about the primal bound, the LP objective value, and the dual multipliers of the variable's bound constraints. Reduced cost fixing can be verified using implicational derivation (with $\primbound < \infty$), see \Cref{sec:appendix3}.

\ourparagraph{Dominated columns.}
Dominated columns presolving~\cite{GamrathKochEtAl2015,Andersen1995} fixes
a variable~$x_k$ to one of its bounds, if there is another variable $x_j$
with $c_j \le c_k$, $a_{*j} \le a_{*k}$, and $x_j$ is integer whenever $x_k$ is. If additionally local (\ie derived) constraints need to be considered,
we assume that they were added to $A\leq b$ for the sake of this paragraph. Such a fixing can be certified by the redundance-based strengthening rule. For the sake of simplicity, assume that~$x_j$ and~$x_k$ are bounded below by $0$ and that $x_j$ has no upper bound; in general, any lower bound and an implied upper bound is possible. Then, a witness is given by $\omega\colon\reals^n \to \reals^n$ with
\begin{equation*}
  \omega(x)_i = \begin{cases}
    x_i, &\text{if } i \notin \{j,k\}, \\
    0 , &\text{if } i = k, \\
    x_i+x_k , &\text{if } i = j.
  \end{cases}
\end{equation*}
To verify that \eqref{eq:redruleimp} of~\Cref{thm:redundance} is satisfied for $\omega$, it is straightforward to check that $\sprod{c}{\omega(x)} \le \sprod{c}{x}$. As $a_{*j} \le a_{*j}$, holds, any constraint satisfied by~$x$ is also satisfied by~$\omega(x)$. This can be automatically verified with an aggregation proof similar as for constraint propagation. Furthermore, $\omega(x)$ trivially satisfies the new constraint $x_k = 0$. In presolving the tree $\bnbtree$ will typically be the initial tree, hence~$\omega(x)\tgeq x$ is a tautology.



\subsection{Symmetry Handling}

Let~$\perm$ be a permutation of~$[n]$.
For~$x \in \reals^n$, let~$\gamma(x) =
(x_{\invperm(1)},\dots,x_{\invperm(n)})$.
A permutation~$\perm$ is a \emph{symmetry} of a MIP~(P) if~$x \in \reals^n$ is feasible for~(P) if
and only if~$\perm(x)$ is feasible, and~$\sprod{c}{x} =
\sprod{c}{\perm(x)}$.
The symmetries of a MIP form a group~$\group$, which is a subgroup of the
symmetric group on~$[n]$, denoted by~$\sym{n}$.

Although this definition of symmetries is natural,
finding all such symmetries is NP-hard~\cite{margot2010symmetry}.
One therefore usually restricts to symmetries that keep the MIP formulation
invariant.
That is, $\perm \in \sym{n}$ is a \emph{formulation symmetry} if there
is a permutation~$\rowperm \in \sym{m}$ of the constraints such
that~$\rowperm(b) = b$, $\perm(c) = c$,
$A_{\inv\rowperm(i),\invperm(j)} = A_{i,j}$ for all~$(i,j) \in [m]
\times [n]$, and the variable types are preserved.
They have the advantage that it is easy to verify if a
pair~$(\perm,\rowperm)$ defines a formulation symmetry, i.e., we can use
them for verifying symmetry reductions.

To handle (formulation) symmetries~$\group$, many techniques have been
discussed~\cite{doornmalenhojny2022cyclicsymmetries,FischettiLiberti2012,HojnyPfetsch2019,Hojny2020,kaibel2008packing,Liberti2012a,Margot2002,Margot2003,OstrowskiAnjosVannelli2015},
and we show that four popular symmetry handling techniques can be derived via our framework:
orbital branching~\cite{OstrowskiEtAl2011}, orbitopal
reduction~\cite{BendottiEtAl2021,KaibelEtAl2011}, Schreier-Sims
cuts (SST cuts)~\cite{LibertiOstrowski2014,Salvagnin2018}, and
lexicographic order based reductions~\cite{SCIP8,Friedman2007}.
As shown in~\cite{doornmalen2023unified}, these methods are
compatible with adding constraints~$\sigma_v(x) \lexgeq
\sigma_v(\perm(x))$ for all~$\perm\in\group$ at the nodes~$v$ of a
consistent branching tree~$\bnbtree$, the entries of~$\sigma_v$ being positive.
Thus, to show that the aforementioned methods fit into our framework, it is
sufficient to derive a representation of~$\sigma_v(x) \lexgeq \sigma_v(\perm(x))$ via linear
inequalities for suitable~$\sigma_v$.

There are many degrees of freedom when building~$\sigma_v$, $v \in V$.
One can select, e.g., (a) some permutation~$\pi$ of~$[n]$ and
assign~$\sigma_v = \pi$ for all $v \in V$, or (b)~if~${v \in V \setminus \{r\}}$ arises from its parent~$u$
by branching on~$x_i$, one can extend~$\sigma_u$ by~$i$ if~$i$ is
not present in~$\sigma_u$, yet.
Approach~(a) creates globally valid symmetry handling inequalities (SHIs);
approach~(b) generates local SHIs that are adapted to the branching history, allowing
to derive reductions earlier, cf.\ the discussion in~\cite{margot2010symmetry}.
As long as the tree is \Cconsistent, all choices of~$\sigma_v$ allow to
handle symmetries.
Specialized symmetry handling methods might require a special structure
of~$\sigma_v$ though, see below.

\ourparagraph{Lexicographic order based reduction.}
A classic approach to handle symmetries is to compute only solutions being
lexicographic maximal representatives among a class of symmetric solutions.
This can be enforced by deriving reductions from the relation~$\sigma_v(x)
\lexgeq \sigma_v(\perm(x))$.
For integer variables with bounded domain contained in~$[L,U]$
of size~$D = U - L$, this relation can be linearly expressed as
\begin{equation}
  \label{eq:FDgen}
  \sum_{i = 1}^{\ell_v} (D + 1)^{\ell_v - i} \cdot \sigma_v(x)_i
  \geq
  \sum_{i = 1}^{\ell_v} (D + 1)^{\ell_v - i} \cdot \sigma_v(\perm(x))_i,
\end{equation}
see, e.g., \cite{Friedman2007}.
\Cref{prop:FD} in \Cref{sec:auxresults} shows that~\eqref{eq:FDgen} can be
derived as local constraint at node~$v$ within our framework via
dominance-based strengthening by selecting~$\omega = \perm$.
Denote the corresponding derived constraint as~$C_v^\perm = [\bradec_v
\rightsquigarrow \eqref{eq:FDgen}]$, where~$\bradec_v = \{B_u : u \text{ is a node on the $r$-$v$-path in~$\bnbtree$}\}$.

Typical reductions that are derived from~$\sigma_v(x)
\lexgeq \sigma_v(\perm(x))$ are variable fixings in the binary case
or variable bound tightenings in the integer case.
Both type of reductions can be encoded as linear constraints.
If a solver finds a lexicographic order based reduction at node~$v$,
e.g., $x_i \leq u_i$ for some~$i \in [n]$ and~$u_i \in \integers$, then
this is the case because there is no~$\bar{x}$ that is feasible for the MIP
and satisfies~$\bar{x}_i \geq u_i + 1$ and~\eqref{eq:FDgen}.
Such a reduction can be derived in our framework by first
deriving~$C_v^\perm$.
In a second step, a CG-proof can be used to show that
\[
  \left\{x \in \integers^p \times \reals^{n-p} \cap \bigcap\bradec_v:
    Ax \leq b,\; x_i \geq u_i + 1,\; x \text{ satisfies~\eqref{eq:FDgen}}\right\}
  =
  \emptyset.
\]
As CG-proofs are compatible with our framework, also~$x_i \leq u_i$
can be derived.

\ourparagraph{Orbitopal reduction.}
While the previously discussed reductions are based on a single
permutation, orbitopal reduction handles symmetries of an entire
specially structured group~$\group$.
Orbitopal reduction assumes the integer variables of a MIP to be
arranged in an~$s\times t$ matrix~$X$ and~$\group$ to operate on the
variables by exchanging the columns of~$X$,
i.e., $\group$ acts on the columns like~$\sym{t}$.
As lexicographic order based reductions, orbitopal reduction derives
variable domain reductions.

In~\cite{doornmalen2023unified}, we discuss that orbitopal reduction is implied by~$\sigma_v(x)
\lexgeq \sigma_v(\perm(x))$ for all~$\perm\in\group$,
if~$\sigma_v$ can be partitioned into consecutive blocks of length~$t$ such
that each block corresponds to exactly one row of~$X$.
Variable domain reductions, e.g., $X_{i,j} \leq
u_{i,j}$ for some~$(i,j) \in [s] \times [t]$ with~$(i,j) \in \sigma_v$, can
be derived as above:
If all variables in~$X$ are bounded and integer,
we derive~\eqref{eq:FDgen} for all~$\perm\in\group$ and use a CG-proof to
show that no MIP solution satisfies~$X_{i,j} \geq u_{i,j} + 1$.
Deriving all inequalities~\eqref{eq:FDgen} is intractable though
as~$\group$ contains~$t!$ permutations.
But, in fact, it is sufficient to derive~\eqref{eq:FDgen} for the~$t-1$
permutations that swap variables of adjacent columns to derive the
reductions of orbitopal reduction, see~\cite{HojnyPfetsch2019}.

\ourparagraph{SST cuts.}
Let~$i \in [n]$ and~$\orb{\group}{i} = \{\perm(i) : \perm\in\group\}$
be its \emph{orbit}.
Simple SHIs are~$x_i \geq x_j$ for~$j \in
\orb{\group}{i}$~\cite{Liberti2012a}, but SHIs for different orbits can
be incompatible.
Compatibility can be ensured by adding these SHIs in
rounds~${i \in [n]}$ for different~$\group_i$ per round~\cite{LibertiOstrowski2014,Salvagnin2018}.
For~$\hat\sigma \in \sym{n}$, round~$i$ adds~$x_{\hat\sigma(i)} \geq x_{j}$ for all $j \in\orb{\group_i}{\hat\sigma(i)}$, and~$\group_i = \{\perm\in\group_{i-1} :
\perm(\hat\sigma(j)) = \hat\sigma(j),\; j \in [i-1] \}$, where~$\group_0 = \group$.
These SHIs (called SST cuts) can be derived from~$\hat\sigma(x) \lexgeq \hat\sigma(\perm(x))$, $\perm\in\group$~\cite{doornmalen2023unified}.


SST cuts can be defined for arbitrary variable types.
Via dominance-based strengthening a slightly
weaker family of constraints, namely~$x_i \geq x_j - \stricteps$ instead
of~$x_i \geq x_j$, can be derived.
If~$\stricteps < 1$ and the involved variables are integral, a rounding
argument can be used to also derive~$x_i \geq x_j$.
For continuous variables, $x_i \geq x_j - \stricteps$
can be gradually made stronger by the epsilon shrinkage rule.

\ourparagraph{Orbital branching.}
To handle symmetries in binary programs, this branching rule
creates two child nodes~\cite{OstrowskiEtAl2011}.
In one node, it enforces~$x_i = 1$ for some~$i \in
[n]$;
in the other node, it enforces~$x_j = 0$ for all~$j \in
\orb{\group'}{i}$ for some subgroup~$\group'$ of~$\group$.
These reductions can be derived from SST cuts~\cite{doornmalen2023unified}
if~$\sigma_v$ is adapted to the branching decisions (cf. the discussion on~$\sigma$ above).
Since SST cuts can be derived in our framework, so can the reductions
by orbital branching.

\subsection{Encoding and Verification}
\label{sec:encoding}

For pure binary programming, the feasibility of implementing and applying the seminal
proof system in~\cite{bogaerts2023certifiedsymbr}, which is the foundation of our work,
has already been demonstrated with great success~\cite{Gocht_2019_veripb,Gocht_et_al2021_isomorphism,ElffersGochtMcCreeshNordstrom_2020, Gocht_et_al2022_CNF, Gocht_et_al2020_clique, Gocht_Nordstroem_2021}.
%
For general MIP, the certificate system~\cite{VIPR} provides a blueprint for
encoding and verifying feasibility
reasoning~\cite{EiflerGleixnerPulaj2022_Asafecomputational,EiflerGleixner2023_Acomputational,EiflerGleixner2023}.
In the following, let us point out differences (and similarities) to both of these that
need to be addressed when approaching an actual implementation of a verifier for MIP.

Both types of constraints in MIP are easily encoded: linear inequalities
and integrality conditions.
The latter may be part of the original model or derived from equations (represented natively or as pairs of inequalities)
via integrality of coefficients and right-hand sides.
But unlike in the pure binary case, locally valid inequalities over unbounded
variables cannot generally be recast as globally valid inequalities.
A suitable verifier hence needs a native notion of implications
$[\mathcal{A} \implication C]$ as in~\cite{VIPR}.
As witnesses~$\omega$, affine maps $x \mapsto Qx+q$, $Q\in\Q^{n\times n}$, $q\in\Q^n$, are both
sufficient and easy to encode.
In many cases, $Q$ will be a sparse permutation matrix, for which only
off-diagonal elements need to be specified.

As described in~\cite{bogaerts2023certifiedsymbr}, a preorder $x \lexgeq y$ on binary
vectors can be elegantly encoded and verified as a system of linear inequalities in $x$
and $y$.
This is not possible in the presence of continuous and unbounded variables.
Instead, the branching tree $\bnbtree=(\bnbnodes,\bnbedges,\bnbdec,\sigma)$ must be
encoded as a generic graph structure with unique identifiers for $\bnbdec_v$ and
$\sigma_v$ attached to the nodes $v\in \bnbnodes$.
To verify $\core$-consistency, checking \eqref{consistent:arborescence},
\eqref{consistent:branching}, \eqref{consistent:dupl}, and \eqref{consistent:growing} is
elementary.
Checking \eqref{consistent:covering} and \eqref{consistent:packing} is also simple when
$\bnbdec_v$ are halfspaces given by linear inequalities.
If the boundedness condition~\eqref{consistent:bounded} does not follow trivially from
bounds on the variables, then it is the task of the certifier to provide
derivations for the boundedness and transfer them to the core before
installing the tree via \Cref{thm:treeexchange}.

The tree is then used to evaluate $\omega(x)\tgeq x$ in \eqref{eq:redruleimp}
and \eqref{eq:symruleimp}.
This can be performed by starting at the root node and diving in the branching tree as
long as $x$ and $\omega(x)$ are known to be contained in the same child.
Here, $x$ is a partial solution with fixed values or tightened bounds according to the
precondition $x \in \bigcap (\core \cup \derived) \land x\not\in C$ and elementary
propagations of these constraints.
If all variables $x$ are fixed, then $\dcn(x,\omega(x))$ and the result of $\omega(x)\tgeq
x$ can always be determined.
For partial solutions, this evaluation may fail, but all techniques discussed in
\Cref{sec:realization} allow for a witness $\omega$ such that it succeeds.
Some of these techniques may require so-called subproofs, \ie preliminary derivations to
strengthen the preconditions, which can afterwards be deleted again from $\derived$.
One example is the inductive derivation of~\eqref{eq:FDgen} 
given in \Cref{prop:FD} in \Cref{sec:auxresults}.

Finally, in its most basic form, a verifier may require the certifier to provide
detailed justification for each derivation.
In an advanced implementation, finding some of these justifications can be automated by
techniques like reverse unit propagation~\cite{GoldbergNovikov2003,ElffersGochtMcCreeshNordstrom_2020},
computing dual multipliers by a rational LP solver~\cite{EiflerGleixner2023}, or
heuristically detecting witnesses from the variable indices in the encoding of
$\sigma$.


\section{Outlook}
\label{sec:outlook}
In \Cref{sec:realization}, we have outlined how the proof system
presented in \Cref{sec:proofsystem,sec:rules} covers a wide range of important
methods implemented in state-of-the-art MIP solvers today, having to
stop short only of a discussion of lifting techniques and
infeasibility analysis, both of which typically require just
implicational reasoning.

Beyond this, we are convinced that this system can be generalized further to
make it more convenient to use and potentially (dis)cover different
algorithms:
\begin{enumerate*}[label=(\alph*)]
\item The $\sigma_v$ may contain duplicate entries, and more generally
could signify any set of dimension-reducing mappings for which a suitably
generalized extension relation~$\sgeq$ holds.
\item The lexicographic order $\lexgeq$ used to compare $\sigma_v(x)$
could become any preorder with an $\stricteps$-strict version
that allows only finitely many strict increases on the bounded image
spaces of the $\sigma_v$.
\item The need to install a fixed branching tree only known
\emph{after} solving could be eliminated by providing an update rule
that certifies a dynamic refinement of $\bnbtree$ \emph{while}
solving.
\end{enumerate*}
We know that such a rule in the style of
\Cref{thm:redundance} can be proven and works for the most common scenario of binary
variable-based branching schemes, but a more general version seems to need
inductive reasoning as in \Cref{thm:dominance} and requires further
research.


\medskip

\noindent
\textbf{Acknowledgements.}
We wish to thank Alexander Hoen, Andy Oertel, and Jakob Nordstr\"om for many patient and
insightful discussions on proof logging for binary programs.

\bibliographystyle{splncs04}
\bibliography{../bibliography}

\clearpage
\appendix

\section{Auxiliary Results and Technical Proofs}

In this appendix, we provide auxiliary results as well as the technical
proofs that we had deferred in the main part of the article.

\subsection{Auxiliary Results}
\label{sec:auxresults}

The first two results state that \Cconsistent branching trees have the
properties that one would expect from a tree generated by a
branch-and-bound algorithm.
That is, the children of a node~$u$ partition the feasible region of the
subproblem at node~$u$, and, for every feasible solution, we find exactly
one leaf of the tree in which the solution is feasible.
In the following results, we use the notation $\bradec_v \define \{B_u : u \text{ is a node on the $r$-$v$-path in~$\bnbtree$}\}$.

\begin{lemma}
  \label{lemma:uniquechild}
  Let $\bnbtree$ be a \Cconsistent branching tree and $u \in \bnbnodes$ with $\children(u)\not=\emptyset$.  Then for all $x \in \bigcap (\core \cup \bradec_{u})$ there is exactly one $v \in \children(u)$ with
  $
  x \in \bigcap (\core \cup \bradec_{v}).
  $
\end{lemma}
\begin{proof}
  Let~$u \in \bnbnodes$ with~$\children(u) \neq \emptyset~$ and let~$x \in
  \bigcap(\core \cup \bradec_u)$.
  By~\eqref{consistent:covering}, there is~$v \in \children(u)$ with~$x \in
  \bigcap(\core \cup \bradec_v)$.
  In particular, for every such~$v \in \children(u)$ we have~$\sigma_u(x)
  \in \sigma_u(\bigcap(\core \cup \bradec_v))\subseteq\sigma_u(B_v)$.
  Uniqueness of~$v \in \children(u)$ thus follows
  from~\eqref{consistent:packing}.
  \qed
\end{proof}

\begin{lemma}\label{lemma:uniqueleaf}
  Let~$\bnbtree$ be a \Cconsistent branching tree.
  Then for all $x \in \bigcap \core$, there is exactly one leaf~$v \in
  \bnbnodes$ with~$x \in \bigcap(\core \cup \bradec_v)$.
\end{lemma}
\begin{proof}
  Because of $\bnbdec_\bnbroot=\reals^n$, $x \in \bigcap (\core \cup \bradec_\bnbroot)$.
  The result follows by induction from \Cref{lemma:uniquechild} and finiteness
  of~$V$ in~$\bnbtree$, see~\eqref{consistent:arborescence}.
  \qed
\end{proof}

A crucial component of an~$(F,f)$-valid configuration as defined in
\Cref{def:validconf} is the relation induced by a branching tree~\bnbtree.
The next results state that this relation and the corresponding strict
relation define a preorder and strict order, respectively.
Afterwards, we also show that a transitivity property between these two
relations exists.

\begin{lemma}
\label{lem:preorder}
  Let~$\bnbtree$ be a \Cconsistent branching tree.
  Then $\tgeq$ defines a preorder and $\tgt$ defines a strict order on $\bigcap\core$.
\end{lemma}
\begin{proof}
  Recall that~$\tgeq$ defines a preorder, if it is reflexive and transitive.
  Reflexivity holds as~$\sigma_{\dcn(x,x)}(x) = \sigma_{\dcn(x,x)}(x)$.
  For transitivity, let~$x,y,z \in \bigcap\core$ with~$y \tgeq x$ and~$z
  \tgeq y$.
  If~$v = \dcn(x,y) = \dcn(x,z) = \dcn(y,z)$,
  transitivity follows by transitivity of~$=$ and~$\lexgteps$.
  In the following, we thus may assume that not all deepest common nodes
  are the same.
  Hence w.l.o.g. we can select pairwise distinct~$i,j,k \in \{x,y,z\}$
  such that~$\dcn(i,k) = \dcn(j,k)$ and~$\dcn(i,j)$ is a proper successor
  of~$\dcn(i,k)$.
  We call this proper successor the \emph{deepest node} and distinguish
  whether~$\dcn(x,y)$, $\dcn(x,z)$, or~$\dcn(y,z)$ is the deepest node:
  \newcommand{\threetree}[3]{%
    \raisebox{-4mm}
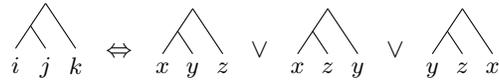
{%
      \begin{tikzpicture}[anchor=north,scale=2,x=1mm, y=-1.5mm, font=\small]
      \draw[-] (0, 0)--(2, 2);
      \draw[-] (0, 0)--(-2, 2);
      \draw[-] (-1, 1)--(0, 2);
      \draw (-2, 2) node {$#1$};
      \draw (0, 2) node {$#2$};
      \draw (2, 2) node {$#3$};
    \end{tikzpicture}}
  }
  \[
  \threetree{i}{j}{k}
  \;\Leftrightarrow\;
  \threetree{x}{y}{z}
  \;\lor\;
  \threetree{x}{z}{y}
  \;\lor\;
  \threetree{y}{z}{x}
  \]

  First, let~$v = \dcn(x,y)$ be the deepest node and~$u = \dcn(x,z)
  = \dcn(y,z)$.
  Then, $y \tgeq x$ implies~$\sigma_v(y) = \sigma_v(x) \lor \sigma_v(y)
  \lexgteps \sigma_v(x)$, which in turn yields that~$\sigma_u(y) = \sigma_u(x)\lor\sigma_u(y)
  \lexgteps \sigma_u(x)$ by~\eqref{consistent:growing} because~$u$ is an
  ancestor of~$v$.
  Together with~$z \tgeq y \Leftrightarrow \sigma_u(z) = \sigma_u(y)\lor\sigma_u(z)
  \lexgteps \sigma_u(y)$, this shows~$z \tgeq x$.

  Second, let~$v = \dcn(x,z)$ be the deepest node and~$u = \dcn(y,z)
  = \dcn(x,y)$.
  Then, $y \tgeq x$ implies~$\sigma_u(y) = \sigma_u(x) \lor \sigma_u(y)
  \lexgteps \sigma_u(x)$, and
  $z \tgeq y$ yields that~$\sigma_u(z) = \sigma_u(y) \lor \sigma_u(z)
  \lexgteps \sigma_u(y)$.
  In fact, we claim that for both relations~$\sigma_u(y) = \sigma_u(x)$
  and~$\sigma_u(z) = \sigma_u(y)$ cannot hold.
  If we are able to show this, $\sigma_u(y) \lexgteps \sigma_u(x)$ and
  $\sigma_u(z) \lexgteps \sigma_u(y)$ imply~$\sigma_v(y) \lexgteps
  \sigma_v(x)$ and~$\sigma_v(z) \lexgteps \sigma_v(y)$, respectively,
  by~\eqref{consistent:growing}.
  Relation~$z \tgeq x$ then follows from transitivity of~$\lexgteps$. 
  
  To prove that~$\sigma_u(y) = \sigma_u(x)$ cannot hold, observe that~$u$
  is not a leaf as~$v$ is a proper successor of~$u$.
  Since~$u = \dcn(x,y)$, solutions~$x$ and~$y$ must be feasible at
  different children of~$u$.
  By~\eqref{consistent:packing}, we conclude~$\sigma_u(y) \neq
  \sigma_u(x)$, and analogously, $\sigma_u(z) \neq \sigma_u(y)$.

  Third, let~$v = \dcn(y,z)$ be the deepest node and~$u = \dcn(x,y)
  = \dcn(x,z)$.
  Then, $z \tgeq y$ implies~$\sigma_v(z) = \sigma_v(y) \lor \sigma_v(z)
  \lexgteps \sigma_v(y)$.
  As in the first case, this yields~$\sigma_u(z) = \sigma_u(y) \lor
  \sigma_u(z) \lexgteps \sigma_u(y)$ by~\eqref{consistent:growing} because~$u$ is an
  ancestor of~$v$.
  Together with~$y \tgeq x$, this shows~$z \tgeq x$, concluding the proof
  that~$\tgeq$ is a preorder on~$\bigcap\core$.
  \smallskip

  To prove that~$\tgt$ is a strict order on~$\bigcap\core$, we need to show
  that it is irreflexive, antisymmetric, and transitive.
  Irreflexivity and antisymmetry follow from irreflexivity and antisymmetry
  of~$\lexgteps$.
  For transitivity, let~$x,y,z \in \bigcap\core$ with~$y \tgt x$ and~$z
  \tgt y$.
  We pursue a similar strategy as above.
  If we have~${v = \dcn(x,y) = \dcn(x,z) = \dcn(y,z)}$, then~$z \tgt x$ follows from
  transitivity of~$\lexgteps$ and that~$\sigma_v$ is used in all comparisons.
  Otherwise, we again distinguish the three cases above.

  First, let~$v = \dcn(x,y)$ be the deepest node and~$u = \dcn(x,z)
  = \dcn(y,z)$.
  Then, $y \tgt x$ yields~$\sigma_v(y)
  \lexgteps \sigma_v(x)$.
  Thus, $\sigma_u(y) = \sigma_u(x) \lor \sigma_u(y) \lexgteps \sigma_u(x)$
  by~\eqref{consistent:growing}.
  From~$z \tgt y$, we get~$\sigma_u(z) \tgt \sigma_u(y)$.
  Combining this shows~${z \tgt x}$.

  Second, let~$v = \dcn(x,z)$ be the deepest node and~$u = \dcn(y,z)
  = \dcn(x,y)$.
  As above, $y \tgt x$ implies~$\sigma_u(y) \lexgteps \sigma_u(x)$,
  and $z \tgt y$ yields~$\sigma_u(z) \lexgteps \sigma_u(y)$.
  Transitivity of~$\lexgteps$ then implies~$\sigma_u(z) \lexgteps \sigma_u(x)$.
  Since~$u$ is an ancestor of~$v$, \eqref{consistent:growing}
  implies~$\sigma_v(z) \lexgteps \sigma_v(x)$, i.e, $z \tgt x$ holds.

  Third, let~$v = \dcn(y,z)$ be the deepest node and~$u = \dcn(x,y)
  = \dcn(x,z)$.
  Then, $z \tgt y$ implies~$\sigma_v(z) \lexgteps \sigma_v(y)$, which again
  implies~$\sigma_u(z) = \sigma_u(y) \lor \sigma_u(z) \lexgteps
  \sigma_u(y)$.
  From~$y \tgt x$, we derive~$\sigma_u(y) \lexgteps \sigma_u(x)$.
  Transitivity
  of~$\lexgteps$ yields~$z \tgt x$.
  Relation~$\tgt$ therefore defines a strict order.
  \qed
\end{proof}

\begin{lemma}\label{lem:doubleTrans}
  Let~$\bnbtree$ be \Cconsistent and let~$x,y,z \in \bigcap\core$
  be such that~$z \tgt y$ and~$y \tgeq x$ hold.
  Then, $z \tgt x$.
\end{lemma}

\begin{proof}
  If~$v = \dcn(x,y) = \dcn(x,z) = \dcn(y,z)$,
  then the chain $z \tgt y \tgeq x$ implies~$\sigma_v(z) \lexgteps \sigma_v(y)  =
  \sigma_v(x) \lor \sigma_v(z) \lexgteps \sigma_v(y)  \lexgteps
  \sigma_v(x)$.
  In the former case, $z \tgt x$ follows immediately, and in the latter
  case it follows from transitivity of~$\lexgteps$.
  In the remainder of the proof, we follow the same strategy and
  terminology as in the proof of \Cref{lem:preorder}.

  First, let~$v = \dcn(x,y)$ be the deepest node and~$u = \dcn(x,z)
  = \dcn(y,z)$.
  Then, $z \tgt y$ implies~$\sigma_u(z) \lexgteps \sigma_u(y)$, and~$y
  \tgeq x$ yields~$\sigma_v(y) = \sigma_v(x) \lor \sigma_v(y) \lexgteps
  \sigma_v(x)$.
  Condition~\eqref{consistent:growing} then implies~$\sigma_u(y) =
  \sigma_u(x) \lor \sigma_u(y) \lexgteps \sigma_u(x)$.
  Combining this with~$\sigma_u(z) \lexgteps \sigma_u(y)$ shows~$z \tgt x$.

  Second, let~$v = \dcn(x,z)$ be the deepest node and~$u = \dcn(y,z)
  = \dcn(x,y)$.
  From~$z \tgt y$, we derive~$\sigma_u(z) \lexgteps \sigma_u(y)$, which
  implies~$\sigma_v(z) \lexgteps \sigma_v(y)$
  by~\eqref{consistent:growing}.
  Moreover, analogously to the proof of \cref{lem:preorder}, we can
  show that~$y \tgeq x$ implies~$\sigma_u(y) \lexgteps \sigma_u(x)$, i.e.,
  $\sigma_u(y) = \sigma_u(x)$ cannot hold.
  Again by~\eqref{consistent:growing}, this implies~$\sigma_v(y) \lexgteps
  \sigma_v(x)$.
  Transitivity of~$\lexgteps$ then shows that~$\sigma_v(z) \lexgteps
  \sigma_v(x)$, and thus, $z \tgt x$ holds.

  Third, let~$v = \dcn(y,z)$ be the deepest node and~$u = \dcn(x,y)
  = \dcn(x,z)$.
  From~$y \tgeq x$, we conclude~$\sigma_u(y) = \sigma_u(x) \lor \sigma_u(y)
  \lexgteps \sigma_u(x)$.
  As above, $\sigma_u(y) = \sigma_u(x)$ cannot hold due to~\eqref{consistent:packing},
  thus $\sigma_u(y) \lexgteps \sigma_u(x)$ follows.
  By~\eqref{consistent:growing}, we obtain $\sigma_v(y) \lexgteps
  \sigma_v(x)$.
  Combining this with~$z \tgt y\Leftrightarrow\sigma_v(z) \lexgteps \sigma_v(y)$, we conclude that~$z \tgt x$ holds.
  \qed
\end{proof}

The next lemma shows that $\core$-consistency of a branching tree is preserved if
the set~$\core$ of core constraints is extended.

\begin{lemma}
  \label{lem:treecons}
  Let $\bnbtree$ be a \Cconsistent branching tree and $C \subseteq \reals^n$.
  Then $\bnbtree$ is also $\core\cup\{C\}$-consistent.
\end{lemma}

\begin{proof}
  \eqref{consistent:arborescence}, \eqref{consistent:branching}, \eqref{consistent:dupl}, \eqref{consistent:growing},
  \eqref{consistent:packing} do not involve \core.
  Because $\bigcap\core$ is additionally intersected with $C$, the boundedness in
  \eqref{consistent:bounded} continues
  to hold.
  Condition \eqref{consistent:covering} follows immediately,
  as~$\bigcap(\core \cup \{C\}) \subseteq \bigcap\core$.
  \qed
\end{proof}

To be able to show that symmetry reductions for MIPs can be derived within
our framework, we require that a linear representation of the lexicographic
order can be derived from our framework, which is achieved by the following
result.

\begin{proposition}
  \label{prop:FD}
  Let~$\configdef$ be an~$(F,f)$-valid configuration with branching tree~$\bnbtree =
  (\bnbnodes, \bnbedges, \bnbdec, \sigma)$ such that all entries
  of~$\sigma_v$, $v \in V$, are positive.
  Let~$\gamma$ be a symmetry of~$(\bigcap\core, g)$ and~$L \leq U$ be integers.
  Let~$v \in \bnbnodes$ be such that~$x_i$, $i \in \sigma_v$, is an integer variable with~$L \leq x_i \leq U$.
  If~$\stricteps \leq 1$, via dominance-based strengthening, one can derive
  the validity of inequality
  \begin{align}
    \label{eq:FD}
    \sum_{i = 1}^{\ell_v} (U-L+1)^{\ell_v - 1} \sigma_v(x)_i
    \geq
    \sum_{i = 1}^{\ell_v} (U-L+1)^{\ell_v - 1} \sigma_v(\gamma(x))_i
  \end{align}
  at~$v$.
  That is, if~$C' = \{x \in \reals^n : x \text{
    satisfies~\eqref{eq:FD}}\}$, one can derive~$C = [\bradec_v
  \rightsquigarrow C']$.
\end{proposition}
We have designed the following proof in such a way that it is
machine-verifiable.
To keep the presentation simple, we do not mention explicitly how the
different steps can be verified.
But the idea is that case distinctions can be implemented as ``subproofs''
in which we add some assumptions characterizing the case.
By using simple arguments in estimations, such as replacing a variable by
their upper or lower bound, these subproofs either prove the desired
statement or lead to a contradiction, showing thus the opposite statement.
\begin{proof}
  To prove that we can derive~$C$ via dominance-based strengthening, we
  need to show that, if~$x \in \bigcap(\core\cup\derived)$ and~$x \notin C$, then
  there exists~$\omega\colon\reals^n\to\reals^n$ with~$g(\omega(x)) \leq
  g(x)$, $\omega(x) \in \bigcap\core$, and~$\omega(x) \tgt x$.
  We claim that~$\omega = \gamma$ serves as a witness.

  Instead of proving the statement just for~$C$, we show a stronger result.
  Let~$\Delta = U - L + 1$ and
  \begin{align}
    \label{eq:FDrec}
    \sum_{i = 1}^{k} \Delta^{k - 1} \sigma_u(x)_i
    &\geq
    \sum_{i = 1}^{k} \Delta^{k - 1} \sigma_u(\gamma(x))_i,
    &&
       k \in [\ell_v].
  \end{align}
  Let~$\hat{C}^j_v = \{x \in \reals^n : x \text{ satisfies~\eqref{eq:FDrec} for } k = j\}$ and~$C^j_v =
  [\bradec_v\rightsquigarrow \hat{C}^j_u]$.
  We claim that we can derive~$C^j_v$ via dominance-based strengthening.
  If we can establish this claim, the assertion follows as~$C =
  C^{\ell_v}_v$.

  To prove this claim, we proceed by induction.
  Our induction hypothesis is that we can show, for all proper
  ancestors~$u$ of~$v$, that we can derive~$C^j_u$ for all~$j \in
  [\ell_u]$, and for node~$v$, we can derive~$C^j_v$ for all~$j \in [k-1]$
  where~$k < \ell_v$.
  Our goal is to show that we can also derive~$C^k_v$.
  Note that the induction base case corresponds to~$k = 1$.
  That is, the proof of the base case and the inductive step is the same.
  Moreover, as~\eqref{eq:FDrec} is invariant under shifting the variable
  domains, we may assume w.l.o.g.\ that~$L = 0$, which will simplify
  notation in the following.

  In the inductive step, by the hypothesis, we may assume that, for all
  proper ancestors~$u$ of~$v$ and~$j \in [\ell_u]$, we have~$C^j_u \in
  \derived$, and~$C^j_v \in \derived$ for all~$j \in [k-1]$.\footnote{This
    assumption is only needed in the proof.
    In practice, one would remove the previously derived constraints
    from~$\derived$ via the deletion rule after adding~$C^k_v$.
  }
  To apply dominance-based strengthening,
  let~$\bar{x} \in \bigcap(\core\cup\derived)$ with~$\bar{x} \notin C^k_v$.
  Since~$\gamma$ is a symmetry of~$(\bigcap\core,g)$, we have~$\gamma(\bar{x})
  \in \bigcap\core$ and~$g(\gamma(\bar{x})) = g(\bar{x})$.
  Thus, it remains to show~$\gamma(\bar{x}) \tgt \bar{x}$.

  Let~$w = \dcn(\bar{x}, \gamma(\bar{x}))$.
  Following the definition of~$\tgt$,
  we need to show that~$\sigma_w(\gamma(\bar{x})) \lexgteps \sigma_w(\bar{x})$.
  Since~$\bar{x} \notin C^k_v$, we have~$\bar{x} \in \bigcap\bradec_v$
  and~$\bar{x} \notin \hat{C}^k_v$.
  By the former, $w$ and~$v$ lie on a common rooted path in~$\bnbtree$.
  We distinguish whether~$w$ is a proper ancestor of~$v$ or not.
  In the latter case, note that~$\sigma_w(\gamma(\bar{x})) \lexgteps \sigma_w(\bar{x})$
  follows from~\eqref{consistent:growing} if we can
  show~$\sigma_v(\gamma(\bar{x})) \lexgteps \sigma_v(\bar{x})$.

  First, suppose~$w$ is not a proper ancestor of~$v$.
  By the induction hypothesis, $C^j_u \in \derived$ for all~$j \in
  [\ell_u]$ and proper ancestors~$u$ of~$v$, as well as~$C^j_v \in
  \derived$ for all~$j \in [k-1]$.
  For all~$j \in [k-1]$,
  this in particular means that~$\bar{x} \in \hat{C}^j_v$, because~$\bar{x} \in
  \bigcap\bradec_v$.
  Suppose that there is one~$j' \in [k-1]$ such that the inequality
  in~$\hat{C}^{j'}_v$ is strictly satisfied.
  Among all such~$j'$, let~$j$ be the one that is minimal.
  Then, since all variables are non-negative by the assumption~$L = 0$,
  \begin{align*}
    \sum_{i = 1}^{k}\Delta^{k-i}\sigma_v(\bar{x})_i
    \geq&
      \sum_{i = 1}^{j-1}\Delta^{k-i}\sigma_v(\bar{x})_i
      +
      \Delta^{k-j} \sigma_v(\bar{x})_j\\
    =&
      \sum_{i = 1}^{j-1}\Delta^{k-i}\sigma_v(\gamma(\bar{x}))_i
      +
      \Delta^{k-j} \sigma_v(\bar{x})_j\\
    \geq&
      \sum_{i = 1}^{j-1}\Delta^{k-i}\sigma_v(\gamma(\bar{x}))_i
      +
      \Delta^{k-j} \sigma_v(\gamma(\bar{x}))_j
      + \Delta^{k-j}\\
    \geq&
      1 + \sum_{i = 1}^{k}\Delta^{k-i}\sigma_v(\gamma(\bar{x}))_i.
  \end{align*}
  Here, the first inequality is due to~$L = 0$.
  The third inequality holds as geometric sums satisfy~$\Delta^{k-j} = 1 + \sum_{i=j+1}^k(\Delta
  - 1)\Delta^i$, and~$\sigma_v(\gamma(\bar{x}))_j
  \leq (U - L) = {\Delta - 1}$.
  The second estimation holds as~$j$ is the smallest index for
  which~\eqref{eq:FDrec} is a strict inequality; its inequality~\eqref{eq:FDrec} thus reduces
  to~$\sigma_v(\bar{x})_j > \sigma_v(\gamma(\bar{x}))_j$, which implies~$\sigma_v(\bar{x})_j
  \geq 1 + \sigma_v(\gamma(\bar{x}))_j$ by integrality.

  This inequality chain, however, is a contradiction to~$\bar{x} \notin C^k_v$.
  Consequently, all inequalities in~$\hat{C}^1_v,\dots,\hat{C}^{k-1}_v$
  must be satisfied by~$\bar{x}$ with equality.
  The inequality in~$\hat{C}^k_v$ then reduces to~$\sigma_v(x)_k \geq
  \sigma_v(\gamma(x))_v$.
  As such, since~$\bar{x} \notin \hat{C}^k_v$, we have~$\sigma_v(\gamma(\bar{x}))_k 
  > \sigma_v(\bar{x})_k$.
  Again, by integrality, this implies
  \[
    \sigma_v(\gamma(\bar{x}))_k
    \geq
    1 + \sigma_v(\bar{x})_k
    \geq
    \stricteps + \sigma_v(\bar{x})_k.
  \]
  As a consequence, $\sigma_v(\gamma(\bar{x})) \lexgteps \sigma_v(\bar{x})$
  follows.

  It remains to consider the case that~$w$ is a proper predecessor of~$v$.
  In this case, $\bar{x}$ and~$\gamma(\bar{x})$ must be feasible at
  different children of~$w$ as~$w$ is the deepest common node.
  By~\eqref{consistent:packing}, $\sigma_w(\gamma(\bar{x}))_{\ell_w} < \sigma_w(\bar{x})_{\ell_w}$
  or~$\sigma_w(\gamma(\bar{x}))_{\ell_w} > \sigma_w(\bar{x})_{\ell_w}$.
  In the former case, exactly the same arguments as above can be used to
  show
  \[
    \sum_{i = 1}^{k}\Delta^{k-i}\sigma_v(\bar{x})_i
    \geq
    \sum_{i = 1}^{k}\Delta^{k-i}\sigma_v(\gamma(\bar{x}))_i,
  \]
  as~$\sigma_v$ is an extension of~$\sigma_w$
  by~\eqref{consistent:growing}, which is a contradiction to~$\bar{x} \notin C^k_v$.
  In the latter case, we find (again using the same arguments)
  that~$\sigma_w(\gamma(\bar{x})) \lexgteps \sigma_w(\bar{x})$.
  After considering all cases, $\gamma(\bar{x}) \tgt \bar{x}$ thus follows.
  \qed
\end{proof}

\subsection{Technical Proofs}

This section provides the missing technical proofs of the rules that allow
to adapt configurations while preserving their~$(F,f)$-validity.

\subsubsection{Proof of \Cref{thm:implic} (Implicational Derivation Rule)}

Let $\config=\configdef$ be the previously given configuration.
\eqref{cond:consistent}, \eqref{cond:obj}, and~\eqref{cond:feasC}  are trivially
preserved since they do not depend on the set of derived constraints.
For \eqref{cond:derive}, let $x \in \bigcap\core$ with $g(x) < \primbound$.
By validity of \config, there exists a $y \in \bigcap (\core \cup \derived)$ with $y
\tgeq x$ and $g(y) \leq g(x) < \primbound$.
It suffices to show $y \in [\assumed\implication C] = \complem{\bigcap\assumed}
\cup C$.
If $y\not\in\bigcap\assumed$, this holds trivially.
If $y\in\bigcap\assumed$, then Condition \eqref{cond:implic} implies
\[
  y\in \bigcap (\core \cup \derived \cup \assumed) \cap \{ x \in \reals^n :
  g(x) < \primbound \} \subseteq C \subseteq [\assumed\implication C].
\]
In either case, \eqref{cond:derive} is satisfied for the updated configuration.
\qed

\subsubsection{Proof of \Cref{cor:resolution} (Resolution Rule)}

First, careful rewriting yields
\begin{align*}
  \bigcap(\core\cup\derived)
  &\subseteq
  \left[\left(\assumed_1 \cup \{A_1\}\right) \implication C_1\right]
  \cap
  \left[\left(\assumed_2 \cup \{A_2\}\right) \implication C_2\right]\\
  &=
  \Big(\,\complem{\bigcap(\assumed_1 \cup \{A_1\})} \cup C_1\Big)
  \cap
  \Big(\,\complem{\bigcap(\assumed_2 \cup \{A_2\})} \cup C_2\Big)\\
  &\subseteq
  \Big(\,\complem{\bigcap(\assumed_1 \cup \{A_1\})} \cup C_1 \cup C_2\Big)
  \cap
  \Big(\,\complem{\bigcap(\assumed_2 \cup \{A_2\})} \cup C_1 \cup C_2\Big)\\
  &=
  \Big(\,\complem{\bigcap(\assumed_1 \cup \{A_1\})}
  \,\cap\,
  \complem{\bigcap(\assumed_2 \cup \{A_2\})} \Big)
  \cup (C_1 \cup C_2)\\
  &=
  \Big(\Big(\,\complem{\bigcap\assumed_1} \,\cup\, \complem{A_1} \Big)
  \,\cap\,
  \Big(\,\complem{\bigcap\assumed_2} \,\cup\, \complem{A_2} \Big) \Big)
  \cup (C_1 \cup C_2).
\end{align*}
By this, every $x\in\bigcap(\core\cup\derived)\subseteq A_1\cup A_2$ is
contained in at least one of the following three sets:
in $C_1\cup C_2$ or in $\complem{\bigcap\assumed_1}$
(if $x\in A_1\setminus(C_1\cup C_2)$) or in $\complem{\bigcap\assumed_2}$ (if $x\in A_2\setminus(C_1\cup C_2)$),
i.e.,
\begin{align*}
  x
  \in
  \complem{\bigcap\assumed_1} \,\cup\, \complem{\bigcap\assumed_2}
  \cup (C_1 \cup C_2)
  =
  \left[(\assumed_1 \cup \assumed_2) \implication (C_1 \cup C_2)\right] \eqqcolon C.
\end{align*}
Hence, $C$ satisfies \eqref{cond:implic} for $\assumed=\emptyset$ and
any value of $\primbound$ (including the value in the current configuration),
and can be added to $\derived$, preserving
validity of the configuration.
\qed

\subsubsection{Proof of \Cref{thm:objbound} (Objective Bound Update Rule)}

Let $\config=\configdef$ be the previously given configuration.
\eqref{cond:consistent} is preserved because it does not involve the objective bound.
Since $\primbound' < \infty$, for \eqref{cond:obj} we need to show that there exists an
$x\in\feasregion$ with $f(x) \leq \primbound'$.
This is a direct consequence of \eqref{cond:feasC} for \config.
\eqref{cond:feasC}   and  \eqref{cond:derive}   continue   to  hold   for  the   updated
configuration because the objective bound becomes stricter and thus their statements become weaker.
\qed

\subsubsection{Proof of \Cref{thm:objfunc} (Objective Function Update Rule)}

Assume that the previously given configuration is $\config=\configdef$.
\eqref{cond:consistent} and \eqref{cond:obj} are not affected by the update of $g$.
For~\eqref{cond:feasC} let $\hat\primbound < \primbound$.

First, assume we have $x\in\feasregion$ with $f(x) \leq \hat\primbound$.
\eqref{cond:feasC} for \config ensures that there also exists $\hat x\in\bigcap\core$ with
$g(\hat x) \leq \hat\primbound < \primbound$, and \eqref{cond:derive} for \config gives us a
$y\in\bigcap(\core\cup\derived)$ with $g(y) \leq g(\hat x) \leq \hat\primbound$.
The condition of the theorem implies that $g'(y)=g(y)$.
All in all, $y\in\bigcap\core$ and $g'(y) \leq \hat z$, so the forward implication in
\eqref{cond:feasC} is preserved.

For the reverse direction, assume we have $x\in\bigcap\core$ with $g'(x) \leq
\hat\primbound < \primbound$.
Then $g(x) = g'(x) \leq \hat\primbound$ holds by the condition of the theorem.
Now \eqref{cond:feasC} for \config ensures that there exists $x\in\feasregion$ with
$f(x) \leq \hat\primbound$.

To show \eqref{cond:derive} for the updated configuration, let $x \in \bigcap\core$ with
$g'(x) < \primbound$.
Again, the condition of the theorem implies $g(x) = g'(x) < \primbound$.
\eqref{cond:derive} for \config yields a $y \in \bigcap (\core \cup \derived)$ with $y
\tgeq x$ and $g(y) \leq g(x)$.
Finally $g'(y) = g(y) \leq g(x) = g'(x)$ follows from the condition of the theorem.
Hence, \eqref{cond:derive} also holds for the updated configuration.
\qed

\subsubsection{Proof of \Cref{thm:epsupdate} (Epsilon Shrinkage Rule)}

\eqref{cond:consistent} is trivially satisfied due to $\stricteps' > 0$.
\eqref{cond:obj} and \eqref{cond:feasC} are not affected when changing $\stricteps$ to
$\stricteps'$.
\eqref{cond:derive} becomes a weaker statement due to $\stricteps' < \stricteps$.
\qed

\subsubsection{Proof of \Cref{thm:transfer} (Transfer Rule)}

Let $\config=\configdef$ be the previously given configuration.
\eqref{cond:consistent} is preserved according to \Cref{lem:treecons},
\eqref{cond:obj} does not depend on \core and \derived, and
\eqref{cond:derive} is preserved because $\bigcap(\core\cup\{C\}) \subseteq
\bigcap\core$ and
\[
  \bigcap\big((\core\cup\{C\})\cup(\derived\setminus\{C\}) \big) =
  \bigcap(\core\cup\derived).
\]
For~\eqref{cond:feasC} let $\primbound' < \primbound$.
First, assume we have $x\in\feasregion$ with $f(x) \leq \primbound'$.
\eqref{cond:feasC} for \config yields $x'\in\bigcap\core$ with $g(x') \leq \primbound' <
z$, and \eqref{cond:derive} for \config gives us a $y\in\bigcap(\core\cup\derived)
\subseteq \bigcap(\core\cup\{C\})$ with $g(y) \leq g(x') \leq \primbound'$.
Hence, the forward implication in \eqref{cond:feasC} is preserved.
For the reverse direction, assume we have $x\in\bigcap(\core\cup\{C\}) \subseteq
\bigcap\core$ with $g(x) \leq \primbound'$.
Now \eqref{cond:feasC} for \config directly yields an $x\in\feasregion$ with $f(x) \leq
\primbound'$.

\qed

\subsubsection{Proof of \Cref{thm:delete} (Deletion Rule)}

Let $\config=\configdef$ be the previously given configuration.
For (a), \eqref{cond:consistent}, \eqref{cond:obj}, and \eqref{cond:feasC} are invariant.
\eqref{cond:derive} becomes weaker, since
$\bigcap(\core\cup\derived')\supseteq\bigcap(\core\cup\derived)$.

For (b) and (c), it suffices to show that $\big( \core', \derived, g, \primbound,
\bnbtree, \stricteps \big)$ is valid; then the validity of the derived constraint deletion
follows from~(a).
The proof for (b) is trivial, since $\bigcap\core'$ and $\bigcap\core$ define the same set.

For proving~(c), note that \eqref{cond:obj} is invariant, and \eqref{cond:consistent} is trivially preserved since all $\sigma_v$ are empty and the tree cannot have nodes with more than one child (and $\epsilon>0$ is unchanged).

The forward direction of \eqref{cond:feasC} is trivial, since
$\bigcap\core\subseteq\bigcap\core'$.
For the reverse direction, we need to show that for any $\hat z < z$, if $x \in \bigcap
\core'$ with $g(x) \le \hat z$ then there exists $y \in \feasregion$ with $f(y) \le
\hat z$.
If $x\in C$, then $x\in\bigcap\core$ and \eqref{cond:feasC} for $\config$ guarantees us a
suitable~$y$.
If $x\not\in C$, then let $\omega$ be the witness provided alongside the application of
the redundance-based strengthening rule.
By~\eqref{eq:redruleimp} applied to~$\big(\core', \emptyset, g, z, \bnbtree, \stricteps \big)$, \ie with $\core=\core'$ and
$\derived=\emptyset$, we know that $\omega(x)\in \bigcap(\core'\cup\{C\})=\bigcap\core$
and $g(\omega(x)) \leq g(x) \leq \hat z$.
Now again, \eqref{cond:feasC} for $\config$ guarantees us a suitable~$y$.

For \eqref{cond:derive}, if $x \in C$ then the condition immediately follows from the validity of the previous configuration $\config$ since $\bigcap(\core'\cup\derived) \supseteq \bigcap(\core\cup\derived)$. If $x \notin C$, then we can once again apply $\omega(x)\in\bigcap\core$ and $g(\omega(x)) \leq g(x) \leq \hat z$, hence the condition immediately follows from \eqref{cond:derive} for $\config$.
\qed

\subsubsection{Proof of \Cref{thm:treeexchange} (Tree Exchange Rule)}

\eqref{cond:consistent} holds because $\bnbtree'$ is assumed to be
\Cconsistent and $\stricteps > 0$.
\eqref{cond:obj} and \eqref{cond:feasC} are independent of the tree.
Due to $\derived=\emptyset$ and the reflexivity of the preorders,
\eqref{cond:derive} holds trivially for $y=x$, independently of the
tree.
\qed

\subsubsection{Proof of \Cref{thm:dimext} (Dimension Extension Rule)}

\eqref{cond:consistent} and \eqref{cond:obj} are invariant under extension of the
dimension.
\eqref{cond:feasC} and \eqref{cond:derive} on the new configuration are equivalent to
\eqref{cond:feasC} and \eqref{cond:derive} on the previous configuration after projecting
out the newly added variable.
\qed

\subsection{Auxiliary examples of verification}
\label{sec:appendix3}
We provide several more detailed examples of automatic verification of cutting planes.

\paragraph*{Chv\'atal-Gomory cuts~\cite{Chvatal1973}.}
Given a valid inequality $\sprod{\alpha}{x} \le \beta$ with $\alpha_i = 0$, for all $i > p$, it is valid to round down all coefficients, while also rounding down the right-hand-side, \ie $\sum_{i=1}^p \down{\alpha_i}x_i \le \down{\beta}$ is valid. This is a special case of a disjunctive cut, where the disjunction is given by $\sum_{i=1}^p \down{\alpha_i}x_i \le \down{\beta} \lor \sum_{i=1}^p \down{\alpha_i}x_i \ge \down{\beta} + 1$ and the greater or equal part of the disjunction contains no points in the feasible region. In the previously existing certificate format for MIP called VIPR~\cite{VIPR}, CG-cuts are directly verified by checking the correctness of the rounding, as well as the integrality requirements. We point out that the interpretation as split cuts allows to verify a CG-cut without knowing the original inequality $\sprod{\alpha}{x} \le \beta$, by solving an auxiliary LP to find the correct Farkas multipliers to show infeasibility of $\sum_{i=1}^p \down{\alpha_i}x_i \ge \down{\beta} + 1$.

\paragraph*{Knapsack cover cuts~\cite{Balas1975}.}
Given a knapsack constraint $\sum_{j=1}^n a_j x_j \le b$, where $x \in \{0,1\}^n, b \in \integers_{>0}, a_j \in \integers_{>0}$ for all~$j \in [n]$, a \emph{cover} is a subset $C \subseteq [n]$ such that $\sum_{j \in C} a_j > b$. The corresponding cover inequality is $\sum_{j \in C} x_j \le |C| - 1$.

Any cover inequality can be verified as a split cut. We define the split disjunction $\sum_{j \in C} x_j \le |C| -1 \lor \sum_{j \in C} x_j \ge |C|$. For the first part of the disjunction, the cover inequality is trivially valid. For the second part, $\sum_{j \in C} x_j \ge |C|$ can be used to derive $x_j \ge 1$ for every~$j \in C$ by a simple aggregation proof. Then $x_j \ge 1$ for all~$j \in C$ can be aggregated using multipliers $a_j$ to derive $\sum_{j \in C} a_j x_j \ge \sum_{j \in C} a_j > b$, which is a trivially verifiable contradiction to the knapsack constraint. Applying the resolution rule certifies the cover inequality as valid.

\paragraph*{Flowcover cuts~\cite{Padberg1985}.}

Flowcover cuts make use of a structure commonly found inside a MIP model, namely single-node flow sets, defined as
\begin{align*}
  T \coloneqq \left\{(x,y) \in \{0,1\}^n \times \reals^n_+ : \sum_{j=1}^n y_j \le b, y_j \le a_jx_j \text{ for } j \in [n] \right\}
\end{align*}
with $0 < a_j \le b$ for all $j \in [n]$. A set $C \subseteq \{1,\ldots,n\}$ is a \emph{flow cover} of $T$ if $\sum_{j \in C} a_j > b$.
We denote by~$(\cdot)^+ = \max\{0, \cdot\}$.
The corresponding flow cover inequality is $\sum_{j \in C} y_j + \sum_{j \in C}(a_j - \lambda)^+(1-x_j) \le b$, where $\lambda = \sum_{j \in C} a_j - b$.

We can verify a flow cover inequality as a split cut with the disjunction $\sum_{j \in C, a_j \ge \lambda}(1-x_j) \le 0 \lor \sum_{j \in C, a_j \ge \lambda}(1-x_j) \ge 1$. In the first case, the inequality reduces to $\sum_{j=1}^n y_j \le b$, which is valid by definition of $T$. In the second case, we can derive the following intermediate inequality:
\begin{align}
  \label{eq:flowcover-inter1}
  & \sum_{j \in C} (a_j-\lambda)^+(1-x_j) = \sum_{j \in C, a_j \ge \lambda} a_j(1-x_j) - \lambda\sum_{j \in C, a_j \ge \lambda}(1-x_j) \\
   &\le \sum_{j \in C} a_j(1-x_j) - \lambda = b - \sum_{j \in C} a_j x_j
   \le \left(b - \sum_{j \in C} a_j x_j\right)^+ \notag
\end{align}

The inequality holds, since $\sum_{j \in C, a_j \ge \lambda}(1-x_j) \ge 1$ by assumption, and since $a_j (1-x_j) \ge 0$.
Then we derive the flow cover inequality using the additional disjunction $\sum_{j \in C} a_jx_j \le b \lor \sum_{j \in C} a_jx_j \ge b + 1$.
In the first case, we find $\sum_{j \in C} y_j \le \sum_{j \in C} a_jx_j = b - (b-\sum_{j \in C} a_jx_j)^+$, and adding the intermediate inequality \eqref{eq:flowcover-inter1} gives the flow cover inequality. In the second case, the inequality $\sum_{j \in C} y_j \le b = b - (b-\sum_{j \in C} a_jx_j)^+$ holds and again adding the intermediate inequality \eqref{eq:flowcover-inter1} gives the flow cover inequality.

\paragraph*{Reduced cost fixing.}

Assume we solved the LP relaxation of (P), with optimal objective value $z_{MIP}$, and assume that we already found an integer feasible solution with objective value $z_I < \infty$. Let $x_j$ be a nonbasic variable at its lower bound, and for the sake of simplicity assume that $x_j \ge 0$ with reduced cost $\bar c_j \neq 0$. Then reduced cost fixing~\cite{conforti2014integer} states that the following inequality is valid:
\begin{equation*}
  x_j \le \down{\frac{z_{MIP}-z_{LP}}{\bar c_j}}
\end{equation*}
Assume that $x_j$ is an integer variable, and for the sake of simplicity, assume that all variables $x_i \ge 0$. The reduction of reduced cost fixing can be proven in the setting of our proof system using the implication rule. We make a simple aggregation proof, adding $\sprod{c}{x} \le z_{MIP}$ and $-(\sprod{y}{A})x \le z_{LP}$, where $y$ are the dual multipliers of the LP solution.
This yields~$(c-\sprod{y}{A})x \le z_{MIP}-z_{LP}$. Since the LP has been solved to optimality, we can aggregate out all variables except~$x_j$ and the inequality stays valid. Dividing by $\bar c_j = c_j - \sprod{y}{A_j}$ yields the desired inequality.

An analogous construction can be used to prove the validity of reduced cost fixing for variables at their upper bound.

\end{document}